\documentclass[12pt]{article}

\usepackage{amssymb}
\usepackage{amsthm}
\usepackage{amsmath}
\usepackage{graphicx}
\usepackage{fullpage}
\usepackage{color}
 \numberwithin{equation}{section}

\theoremstyle{plain}
\newtheorem{thm}{Theorem}[section]
\newtheorem{cor}[thm]{Corollary}

\newtheorem{lem}[thm]{Lemma}
\newtheorem{prop}[thm]{Proposition}

\theoremstyle{definition}
\newtheorem{defn}[thm]{Definition}

\theoremstyle{remark}

\newtheorem{rem}[thm]{Remark}

\newcommand{\N}{\mathbb{N}}
\newcommand{\R}{\mathbb{R}}

\newcommand{\I}{\infty}

\newcommand{\ffint}{\iint_{Q_r} \!\!\!\!\!\!\!\!\!\!\!\!\!\!\text{-----}}

\newcommand{\fffint}{\iint_{Q_1} \!\!\!\!\!\!\!\!\!\!\!\!\!\!\text{-----}}
\newcommand{\fthetaint}{\iint_{Q_{\theta_0}} \!\!\!\!\!\!\!\!\!\!\!\!\!\!\text{-----}}
\newcommand{\gfint}{\int_{B_r} \!\!\!\!\!\!\!\!\! -}
\newcommand{\bp}{\begin{proof}[\ensuremath{\mathbf{Proof}}]}
\newcommand{\bs}{\begin{proof}[\ensuremath{\mathbf{Solution}}]}
\newcommand{\ep}{\end{proof}}

\begin{document}


\title{Partial regularity of weak solutions of the viscoelastic Navier-Stokes equations with damping }

\author{Ryan Hynd\thanks{This material is based upon work supported by the National Science Foundation under Grant No. DMS-1004733.}\\
Courant Institute of Mathematical Sciences\\
New York University\\
251 Mercer Street\\
New York, NY 10012-1185 USA}  

\maketitle

\begin{abstract}
We prove an analog of the Caffarelli-Kohn-Nirenberg theorem for weak solutions of a system of PDE that model a viscoelastic fluid in the presence of an energy damping mechanism.  
The system was recently introduced as a possible method of establishing the global in time existence of weak solutions of the well known Oldroyd system.
\end{abstract}

\tableofcontents




\section{Introduction}
The Oldroyd model for an incompressible, viscoelastic fluid is governed by the following system of equations
\begin{equation}\label{oldSystem}
\begin{cases}
\hspace{.06in}\partial_tu + (u\cdot \nabla)u  = \Delta u - \nabla p +\nabla\cdot FF^t\\
\partial_t F +(u\cdot \nabla)F  =\nabla u F\\
\hspace{.73in}\nabla \cdot u = 0
\end{cases}.
\end{equation}
Informally, we refer to \eqref{oldSystem} as the {\it viscoelastic Navier-Stokes equations}.  This system of PDE is always assumed to be satisfied in an open subset of spacetime $\R^3\times \R.$ 
At a point $(x,t)$ in spacetime,  $u=u(x,t)\in \R^3$ represents the fluid's velocity, $ p=p(x,t)\in \R$ represents the fluid's pressure, 
and $F=F(x,t)\in \R^{3\times 3}$ represents the local deformation of the fluid.  The associated energy law for any smooth solution $(u,p,F)$ on $\R^3\times (0,\infty)$ that vanishes
rapidly enough as $|x|\rightarrow \infty$ is
$$ 
\frac{d}{dt}\int_{\R^3}\left\{\frac{|u(x,t)|^2}{2}+\frac{|F(x,t)|^2}{2}\right\}dx =- \int_{\R^3}|\nabla u(x,t)|^2dx
$$
for $t>0.$

\par Solutions of initial value problems associated to \eqref{oldSystem} have been studied extensively.  For instance, the short time existence of a smooth solution and the global existence of a smooth 
solution that is initially small (in an appropriate norm) has been established in various settings \cite{LZ, LLZ}.  However, it is not known if solutions with smooth initial and boundary data develop 
singularities or even if some meaningful type of weak solutions exist globally in time.  

\par In pursuing the former problem, the authors of \cite{LLZ} introduced the following system as a way of approximating solutions of
\eqref{oldSystem}
\begin{equation}\label{mainSystem}
\begin{cases}
\hspace{.06in}\partial_tu + (u\cdot \nabla)u  = \Delta u - \nabla p +\nabla\cdot FF^t\\
\partial_t F +(u\cdot \nabla)F  = \mu \Delta F+\nabla u F \\
\hspace{.73in}\nabla \cdot u = 0
\end{cases}
\end{equation}
for a parameter $\mu>0.$ The associated energy law for any smooth solution $(u,p,F)$ of \eqref{mainSystem} on $\R^3\times (0,\infty)$, that vanishes
rapidly enough as $|x|\rightarrow \infty$, is
$$
\frac{d}{dt}\int_{\R^3}\left\{\frac{|u(x,t)|^2}{2}+\frac{|F(x,t)|^2}{2}\right\}dx =- \int_{\R^3}\left(\mu |\nabla F(x,t)|^2+|\nabla u(x,t)|^2\right)dx
$$
for $t>0.$  Therefore, the presence of $\mu $ acts to create energy dissipation, and so we interpret $\mu $ as a {\it damping parameter}  
and call \eqref{mainSystem} the {\it viscoelastic Navier-Stokes equations with damping.}

\par It is not difficult to establish the existence of a global in time weak solution of \eqref{mainSystem} analogous to that of the Leray-Hopf solutions of the incompressible Navier-Stokes 
equations.  Unfortunately, standard weak convergence methods do not allow one to pass to the limit as $\mu\rightarrow 0^+$ to generate weak solutions of \eqref{oldSystem}; see the end of section 2 in reference \cite{LLZ} for more
this.  Nevertheless, the system \eqref{mainSystem} is itself of interest and is the topic of study in our work. 

\par A fundamental simplification that we will make in our analysis is that we consider solutions $(u,p,F)$ of \eqref{mainSystem} that additionally satisfy the equation
\begin{equation}\label{divFtranspose}
\nabla\cdot F^t=0.
\end{equation}
That is, a standing assumption that we shall make is that the deformation $F$ has divergence free columns.  This assumption is motivated from taking the divergence of the second equation in \eqref{mainSystem} which yields
the following  transport equation
$$
\partial_t (\nabla\cdot F^t)+ (u\cdot \nabla)(\nabla\cdot F^t) = \mu \Delta (\nabla\cdot F^t).
$$
The above PDE formally implies that if \eqref{divFtranspose} holds at some instance of time, then it will hold at all later times.  Therefore, we believe that our results below will be pertain to solutions of \eqref{mainSystem} that initially satisfy \eqref{divFtranspose}.

\par As our results do not change qualitatively as $\mu>0$ is varied, we set 
$$
\mu =1
$$
in our analysis of weak solutions of \eqref{mainSystem}.  Our main result is the analog for \eqref{mainSystem} of the Caffarelli-Kohn-Nirenberg
theorem for the incompressible Navier-Stokes equations \cite{CKN}.  The statement of this theorem involves the concept of a weak solution, which we will define in the 
next section, and the concept of the singular set of a solution $(u,p,F)$.  The {\it singular set} corresponding to a solution $(u,p,F)$ is defined as the set of points $(x,t)$ in the domain of $(u,p,F)$
for which either $u$ of $F$ is not H\"{o}lder continuous in any neighborhood of $(x,t).$ Any point not belonging to the singular set of $(u,p,F)$ is a {\it regular point}.

\begin{thm}\label{mainthm}
There is a universal constant $\epsilon>0$ such that if $(u,p,F)$ is a  weak solution of \eqref{mainSystem} and if
\begin{equation}\label{CKNcond}
\limsup_{r\rightarrow 0^+}\frac{1}{r}\iint_{Q_r(x,t)}\left\{|\nabla u|^2 + |\nabla F|^2\right\}dyds<\epsilon,
\end{equation}
then $u$ and $F$ are H\"{o}lder continuous on some neighborhood of $(x,t)$. In particular, $(x,t)\notin S.$ 
\end{thm}
In the limit above \eqref{CKNcond}, and in this work,  
$$
Q_r(x,t):=B_r(x)\times (t-r^2/2,t+r^2/2)\subset \R^3\times\R
$$
is a parabolic cylinder of radius $r$ centered at $(x,t)$.  Using standard covering arguments (see in particular section 6 of \cite{CKN}), we have the following corollary. 
\begin{cor}
Assume that $(u,p,F)$ be a suitable weak solution of \eqref{mainSystem} on an open subset of $\R^3\times \R$ and let $S$ be the singular set of $u$ and $F$. Then ${\cal P}^1(S)=0$, where ${\cal P}^1$ denotes one-dimensional parabolic Hausdorff measure on $\R^3\times\R.$
\end{cor}
In proving Proposition \ref{mainthm}, we employ the blow-up/compactness method introduced by Lin \cite{L} that lead to a relatively simple proof of the Caffarelli-Kohn-Nirenberg theorem.   We also borrow many of the 
ideas from the very clear account on these topics given by Ladyzhenskaya and Seregin \cite{LS}.  The organization of this paper is as follows. In section \ref{WS}, we define suitable weak solutions and establish an important compactness result for these solutions;  a fundamental corollary of this compactness is a ``decay" or ``blow-up" lemma. In section \ref{HC}, we use the decay lemma to deduce a local condition that if satisfied implies solutions are locally H\"{o}lder continuous.  Finally, in section \ref{ABCDest}, we show that \eqref{CKNcond} implies that our local regularity condition is satisfied, which in turn furnishes a proof Theorem \ref{mainthm}.

\section{Weak solutions}\label{WS}
As mentioned above, we set $\mu=1$ in equation \eqref{mainSystem} and study the system of PDE 
\begin{equation}\label{System1}
\begin{cases}
\hspace{.06in}\partial_tu + (u\cdot \nabla)u  = \Delta u - \nabla p +\nabla\cdot FF^t\\
\partial_t F +(u\cdot \nabla)F  = \Delta F+\nabla u F \\
\hspace{.73in}\nabla \cdot u = 0
\end{cases}.
\end{equation}
Sometimes it will be beneficial for us to write equation \eqref{System1}  in terms of the columns of the matrix valued mapping $F$. Setting $F_j:=Fe_j$  for $j=1,2,3$, we have
from the assumption \eqref{divFtranspose}
\begin{equation}
\nabla\cdot FF^t=\sum^3_{k=1}(F_k\cdot \nabla)F_k. \nonumber
\end{equation}
In particular, \eqref{System1} and  \eqref{divFtranspose} can be rewritten together as

\begin{equation}\label{System1Column}
\begin{cases}
\hspace{.15in}\partial_tu + (u\cdot \nabla)u  = \Delta u - \nabla p +\sum^3_{k=1}(F_k\cdot \nabla)F_k  \\
\partial_t F_j +(u\cdot \nabla)F_j  =\Delta F_j+(F_j\cdot\nabla) u ,\quad j=1,2,3\\
\hspace{.8in}\nabla \cdot u = \nabla \cdot F_j=0
\end{cases}.
\end{equation}
As Theorem \ref{mainthm} is local, we only consider solutions on the unit cylinder $Q_1:=Q_1(0,0).$   Before pursuing the analysis of solutions, let us make some basic observations that 
will motivate the definition of suitable weak solutions and other ideas to follow.  
\\
\\
{\bf Scale invariance}. If $(u,p,F)$ is a solution of \eqref{System1} on $Q_r(x_0,t_0)$, then $(u^\lambda, p^\lambda, F^\lambda)$ is a solution on 
$Q_{r/\lambda}(0,0)$ for $\lambda>0$, where
\begin{equation}\label{ScaleInv}
\begin{cases}
u^\lambda(x,t)=\lambda u(x_0+\lambda x,t_0+\lambda^2 t)\\
u^\lambda(x,t)=\lambda^2 p(x_0+\lambda x,t_0+\lambda^2 t)\\
F^\lambda(x,t)=\lambda F(x_0+\lambda x,t_0+\lambda^2 t)
\end{cases}.
\end{equation}
{\bf Local energy identity}. If $(u,p,F)$ is a smooth solution of \eqref{System1} on $Q_1$ and $\phi\in C^\I_c(Q_1)$, then
\begin{align*}
\frac{d}{dt}\int_{B_1}\phi\left(\frac{|u|^2}{2}+\frac{|F|^2}{2}\right)dx + \int_{B_1}\phi\left(|\nabla u|^2+|\nabla F|^2\right)dx &= \int_{B_1}\left\{(\phi_t+\Delta \phi)\left(\frac{|u|^2}{2}+\frac{|F|^2}{2}\right) \right.\\
&\left. -u\otimes\nabla\phi\cdot FF^t + \left(\frac{|u|^2}{2}+\frac{|F|^2}{2}+p\right)u\cdot \nabla \phi \right\}dx
\end{align*}
for $t\in (-1/2,1/2).$
\\
\\
{\bf A spacetime $L^{10/3}$ bound on $u,F$}. From the local energy identity above, we expect solutions $(u,p,F)$ to satisfy
\begin{equation}\label{SpaceTimeIntegral}
\sup_{-1/2\le t\le 1/2}\int_{B_1}|u|^2 +|F|^2dx + \iint_{Q_1}|\nabla u|^2 + |\nabla F|^2dxdt\le C.
\end{equation}
An application of the interpolation estimate 
\begin{equation}\label{InterpolationINEQ}
|v|_{L^r(\Omega)}\le C\left\{|v|_{L^2(\Omega)}^\alpha |\nabla v|_{L^2(\Omega)}^{1-\alpha} + \frac{|v|_{L^2(\Omega)}}{|\Omega|^{1/2-1/r}}\right\}
\end{equation}
for $v\in H^1(\Omega; \R^3)$, where
$$
\alpha = \frac{3}{r}-\frac{1}{2}, \quad 2\le r\le 6,
$$
provides the bound
\begin{equation}\label{uFtenthirds}
\iint_{Q_1}| u|^{10/3} + | F|^{10/3}dxdt\le C_1.
\end{equation}
Here $C_1$ only depends on the constant $C$ in \eqref{SpaceTimeIntegral}. 
\\
\\
{\bf A spacetime $L^{5/3}$ bound on $p$}.  Taking the divergence of equation \eqref{System1} gives 
\begin{equation}\label{pEq}
-\Delta p = \nabla\cdot \left[(u\cdot \nabla)u - \sum^{3}_{k=1}(F_k\cdot \nabla)F_k\right].
\end{equation}
Note that 
\begin{equation}\label{pbounds1}
|(u\cdot\nabla)u|_{L^{15/14}(B_1)}\le |\nabla u|_{L^2(B_1)} | u|_{L^{30/13}(B_1)}.
\end{equation}
By the interpolation inequality \eqref{InterpolationINEQ} with 
$r=30/13$ and $\alpha=4/5$
\begin{equation}\label{pbounds2}
| u|_{L^{30/13}(B_1)}\le C\left\{|\nabla u|_{L^2(B_1)}^{1/5}+1\right\}.
\end{equation}
With \eqref{pbounds1} and \eqref{pbounds2}, we have
$$
|(u\cdot\nabla)u|^{5/3}_{L^{15/14}(B_1)}\le C\left\{|\nabla u|_{L^2(B_1)}^{2}+1\right\}
$$
and likewise
$$
|(F_j\cdot\nabla)F_j|^{5/3}_{L^{15/14}(B_1)}\le C\left\{|\nabla F_j|_{L^2(B_1)}^{2}+1\right\}, \quad j=1,2,3.
$$
Since, 
$$
\left(\frac{15}{14}\right)^*=\frac{5}{3},
$$
we also have by the Sobolev inequality and basic PDE estimates following from equation \eqref{pEq}
\begin{align}\label{pfivethirdsEst}
\iint_{Q_\theta}|p|^{5/3}dxdt & \le C\int^{\theta^2/2}_{-\theta^2/2}|\nabla p|^{5/3}_{L^{15/14}(B_\theta)}dt \nonumber \\
&\le  C\int^{1/2}_{-1/2}\left(|(u\cdot\nabla)u|^{5/3}_{L^{15/14}(B_1)}+\sum^3_{j=1}|(F_j\cdot\nabla)F_j|^{5/3}_{L^{15/14}(B_1)}+1\right)dt \nonumber \\
&\le C\iint_{Q_1}\left(|\nabla u|^2 +\sum^3_{j=1}|\nabla F_j|^2+1\right)dxdt
\end{align}
for any $\theta\in (0,1)$.  Thus $p\in L^{5/3}_{\text{loc}}(Q_1).$ 
\\
\par The above observations motivate the following definition of weak solutions of the system \eqref{System1}.  This definition is of course 
also consistent with the notion of (suitable) weak solutions presented in the original work by Caffarelli, Kohn and Nirenberg \cite{CKN}. 

\begin{defn}\label{SuitWeakSoln}
$(u,p,F)$ is a {\it weak solution} of \eqref{System1} on $Q_1$ provided \\\\
$(i)$ $u,F\in L^3(Q_1)$, and $p\in L^{3/2}(Q_1)$, \\\\
$(ii)$ equation \eqref{System1} holds in the sense of distributions on $Q_1$, and \\\\
$(iii)$ for each $\phi\in C^\I_c(Q_1)$, with $\phi\ge 0$
\begin{align}\label{NSEenergy}
\int_{B_1\times\{t\}}\phi\left(\frac{|u|^2}{2}+\frac{|F|^2}{2}\right)dx &+ \int^t_{-1/2}\int_{B_1}\phi\left(|\nabla u|^2+|\nabla F|^2\right)dxds \le  \int^t_{-1/2}\int_{B_1}\left\{(\phi_t+\Delta \phi)\left(\frac{|u|^2}{2}+\frac{|F|^2}{2}\right)\right.\nonumber\\
&\left.  - u\otimes\nabla\phi\cdot FF^t+ \left(\frac{|u|^2}{2}+\frac{|F|^2}{2}+p\right)u\cdot \nabla \phi \right\}dxds
\end{align}
$t\in [-1/2,1/2]$.
\end{defn}

\par Our first order of business is to establish that a bounded sequence of weak solutions has a convergent subsequence
whose limit is again a weak solution.  We will not use this result directly altough the ideas that go into proving this 
result will be essential to our proof of Theorem \ref{mainthm}.    To this end, we start by quoting a relatively standard compactness lemma. 


\begin{lem}\label{LionsCompactness}
Let $X_0, X_1, X_2$ denote Banach spaces, with $X_0$ and $X_2$ reflexive, that satisfy
$$
X_0\subset X_1\subset X_2.
$$
Also suppose that the embedding of $X_0$ into $X_1$ is compact and the embedding of $X_1$ into $X_2$ is continuous. Let $p,q \in (1,\infty)$ and assume that 
$(u_k)_{k\in \N} \in L^p([0,1]; X_0)$ is a bounded sequence such that each $u_k$ has a weak derivative $u_k'$ and the sequence 
$$
(u'_k)_{k\in \N} \in L^q([0,1]; X_2)
$$ 
is also bounded. Then there is a subsequence of $u_k$ converging strongly in $L^p([0,1], X_1)$.
\end{lem}
\begin{rem}
We can replace $[0,1]$ in the statement above with any compact interval of $\R.$
\end{rem}
We refer the reader to Theorem 2.1, section III of \cite{TM} for a detailed proof of the above lemma.  Its primary application is the following theorem.

\begin{thm}\label{CompactSoln}
Suppose that $(u^k,p^k,F^k)_{k\in \N}$ is a sequence of weak solutions of \eqref{System1} on $Q_1$ satisfying 
\begin{equation}\label{LPestUPF}
\iint_{Q_1}\left\{|u^k|^3+|p^k|^{3/2}+|F^k|^3\right\}dxdt\le C
\end{equation}
for all $k\ge 1$ and some constant $C\ge 0$.  Then there is a subsequence $(u^{k_j},p^{k_j},F^{k_j})_{j\in \N}$ and 
$(u,P,F)$, such that 
$$
\begin{cases}
u^{k_j}\rightarrow u \quad \text{in }\quad L^q_{\text{loc}}(Q_1)  \\
p^{k_j}\rightharpoonup P \quad \text{in }\quad L^{3/2}(Q_1)  \\
F^{k_j}\rightarrow F\quad \text{in }\quad L^q_{\text{loc}}(Q_1)  
\end{cases}
$$
for $1\le q<10/3.$  Moreover, $(u,P,F)$ is a weak solution of \eqref{System1} on $Q_1$.
\end{thm}

\begin{rem}
It is straightforward to adapt the proof below to build globally defined weak solutions of a large class of initial value problems associated with  \eqref{System1}.  One may 
use a Galerkin-type approximation, for example.
\end{rem}

\begin{proof} First, we select the weak limit $P$ of the sequence $(p^k)_{k\in \N}$ and without any loss of generality assume $p^k\rightharpoonup P$ in $L^{3/2}(Q_1)$ as $k\rightarrow \infty.$  Next, we observe that
by \eqref{System1Column}, we have that 
$$
\partial_t u^k = \nabla\cdot \left(-u^k\otimes u^k +\nabla u^k - p^kI +F^k(F^k)^t\right) 
$$
and  
$$
\partial_t F^k_j = \nabla\cdot \left(- F^k_j\otimes u^k  +u^k\otimes F^k_j+ \nabla F^k_j\right), \quad j=1,2,3
$$
belong to the space
$$
L^{3/2}((-1/2,1/2); W^{-1,3}(B_1)).
$$
By \eqref{LPestUPF},  $\partial_t u^k $ and $\partial_t F^k$ are in fact bounded in this space
for all $k\ge 1$.  

\par For a fixed $\theta \in (0,1)$, we note that as $(u^k,p^k,F^k)$ satisfies the local energy estimate \eqref{NSEenergy}
\begin{equation}\label{GradTHETA}
\iint_{Q_\theta}|\nabla u^k|^2+|\nabla F^k|^2dxdt\le C
\end{equation}
for some $C=C(\theta)$. These observations lead to the choice of exponents
$$
p:=2\quad\text{and}\quad q:=3/2
$$
and Banach spaces
$$
\begin{cases}
X_0:=H^1(B_\theta)\\
X_1:=L^2(B_\theta)\\
X_2:= W^{-1,3}(B_\theta)
\end{cases}.
$$
\par  
By our observations above, the hypotheses of the previous lemma are satisfied with the sequences $(u^k)_{k\in \N}$ and $(F^k)_{k\in \N}$  (and exponents and Banach spaces as indicated above).  Hence, some subsequence of $(u^k)_{k\in \N}$, $(F^k)_{k\in \N}$ converge in $L^2(Q_\theta)$ to some $u_\theta,F_\theta$. 
By the bound \eqref{GradTHETA}, $u^k,F^k$ are bounded in $L^{10/3}(Q_\theta)$ (recall the estimate \eqref{uFtenthirds}). Consequently, the interpolation of Lebesgue spaces implies $(u^k,F^k)\rightarrow (u_\theta,F_\theta)$ as $k\rightarrow$ in $L^q(Q_\theta)$ for $1\le q<10/3.$ As $\theta\in (0,1)$ was arbitrary, we can employ a routine diagonalization argument to construct a $u,F\in L^{10/3}_{\text{loc}}(Q_1)$ such that a subsequence of $u^k$ and $F^k$ converge respectively to $u$ and $F$ in $L^q_{\text{loc}}(Q)$ for $1\le q<10/3.$
It is now immediate to pass to the limit as $k\rightarrow \infty$ and show that $(u,P,F)$ is a solution of \eqref{System1} in the distributional sense and that \eqref{NSEenergy} holds.  Therefore, $(u,P,F)$ is a weak solution of \eqref{System1} on $Q_1$.
\end{proof}
Another application of the compactness lemma is the following ``blow-up" or ``decay" lemma; these names come from the fact that the lemma's proof uses a rescaling and blow-up argument, while the conclusion involves a type of decay. This is arguably the most important step in the proof of Theorem \ref{mainthm}.  These ideas originated with the groundbreaking work of Lin \cite{L}. However, the specific approach we use here follows closely the reinterpretation by  Seregin and Ladyzhenskaya \cite{LS}.  

\begin{lem}\label{DecayLem}
Let $(u,p,F)$ be a weak solution on $Q_1$ and set 
$$
E(x,t,r):= \left( \ffint |u-u_{Q_r}|^3dyds\right)^{1/3}+r\left(\ffint |p-p_{B_r}(t)|^{3/2}dyds \right)^{2/3}+ \left( \ffint|F-F_{Q_r}|^3dyds \right)^{1/3}
$$
for $Q_r=Q_r(x,t)$ and $B_r=B_r(x)$, where 
$$
u_{Q_r}:=\ffint u(y,s)dyds\quad \text{and}\quad p_{B_r}(s):=\gfint p(y,s)dy.
$$ For each  
$$
0<\theta <1/2 \quad \text{and}\quad M>0
$$
there are positive numbers $\epsilon_1, R_1, c_1$ such that if 
$$
\begin{cases}
(i) \quad Q_r(x,t)\subset Q_1, \; r\le R_1\\\\
(ii) \quad |u_{Q_r(x,t)}|+ |F_{Q_r(x,t)}|\le M\\\\
(iii) \quad E(x,t,r)\le \epsilon_1
\end{cases},
$$
then 
$$
E(x,t,\theta r )\le c_1 \theta^{2/3}E(x,t,r).
$$
Moreover, $c_1$ can be chosen independent of $\theta$. 
\end{lem}

\begin{rem}
In particular we can arrange $c_1\theta^{2/3}$ to be small, so this is a type of decay lemma. 
\end{rem}

\begin{proof}
1. We argue by contradiction. Suppose the statement of the lemma is false. Then there are sequences
$$
\begin{cases}
(x_k,t_k)\in Q_1\\
\epsilon_k\rightarrow 0\\
r_k\rightarrow 0\\
c_k\equiv c_1 \; (\text{which will be chosen sufficiently large below})
\end{cases}
$$
with 
\begin{equation}\label{SeqBounds}
\begin{cases}
(i) \quad Q_{r_k}(x_k,t_k)\subset Q_1,\\\\
(ii) \quad | u_{Q_{r_k}(x_k,t_k)}|+ | F_{Q_{r_k}(x_k,t_k)}|\le M_0\\\\
(iii) \quad E(x_k,t_k,r_k)= \epsilon_k
\end{cases},
\end{equation}
and 
\begin{equation}\label{WhatWeDontWant}
E(x_k,t_k,\theta_0 r_k)>c_1\theta_0^{2/3}E(x_k,t_k,r_k).
\end{equation}
We set
$$
\begin{cases}
u^k(y,s):=\epsilon_k^{-1}\left(u(x_k+r_k y, t_k + r_k^2 s)- u_{Q_{r_k}(x_k,t_k)}\right) \\\\
p^k(y,s):=r_k\epsilon_k^{-1}\left(p(x_k+r_k y, t_k + r_k^2 s)- p_{B_{r_k}(x_k)}(t_k+r_k^2s)\right) \\\\
F^k(y,s):=\epsilon_k^{-1}\left(F(x_k+r_k y, t_k + r_k^2 s)- F_{Q_{r_k}(x_k,t_k)}\right)
\end{cases}, \quad (y,s)\in Q_1,
$$
and observe that $(u^k,p^k, F^k)$ is a weak solution (defined analogously as in Definition \ref{SuitWeakSoln}) of the PDE

\begin{equation}\label{PDEsequence}
\begin{cases}
\hspace{.07in}\partial _tu^k + ((a^k+\epsilon_k u^k)\cdot \nabla)u^k =\Delta u^k - \nabla p^k + \sum^3_{j=1}((b^k_j +\epsilon_k F^k_j)\cdot\nabla) F^k_j\\
\partial _tF^k_j + ((a^k+\epsilon_k u^k)\cdot \nabla)F^k_j  =\Delta F^k_j  + ((b^k_j +\epsilon_k F^k_j)\cdot\nabla) u^k\\
\hspace{1.5in}\nabla \cdot u^k=\nabla\cdot F^k_j=0
\end{cases}, (y,s)\in Q_1.
\end{equation}
for $j=1,2,3.$ Here 
$$
a^k:= u_{Q_{r_k}(x_k,t_k)} \quad \text{and}\quad b^k_j:=(F_j)_{Q_{r_k}(x_k,t_k)} 
$$ 
are bounded sequences in $\R^3$ for each $j=1,2,3$ by \eqref{SeqBounds}.  Moreover, the sequence $(u^k,p^k, F^k)$ satisfy the integral bounds
\begin{equation}\label{IntegralboundsSequence}
 \left(\fffint|u^k|^3 \right)^{1/3}+\left( \fffint|p^k|^{3/2} \right)^{2/3}+ \left(\fffint|F^k|^3 \right)^{1/3}=\frac{E(x_k,t_k,r_k)}{\epsilon_k}=1
\end{equation}
and the generalized energy inequality          
{\small \begin{align}\label{EnergySequence}
\int_{B_1\times\{s\}}\phi\left(\frac{ |u^k|^2}{2}+\frac{ |F^k|^2}{2}\right)dx + \int^{s}_{-1/2}\int_{B_1}\phi\left(|\nabla u^k|^2+|\nabla F^k|^2\right)dxds \le  \int^{s}_{-1/2}\int_{B_1}\left[
\left(\frac{|u^k|^2}{2}+\frac{|F^k|^2}{2}\right)(\phi_t+\Delta \phi) \right. \nonumber \\
\left. + \left\{  -\sum^3_{j=1}u^k\cdot F^k_j(b^k_j+\epsilon_k F^k_j)+\left(\frac{|u^k|^2}{2}+\frac{|F^k|^2}{2}\right)(a^k+\epsilon_k u^k)+p^k u^k\right\}\cdot \nabla \phi \right]dxds
\end{align}}
for $\phi\in C^\I_c(B), \phi\ge 0$ and $s\in (-1/2,1/2).$ 

\par 2. Arguing as we did in the proof of the Theorem \ref{CompactSoln} one checks that the sequence $(u^k)_{k\in \N}$ and $(F^k)_{k\in \N}$ satisfy the hypotheses of the Lemma \ref{LionsCompactness}. Moreover, we conclude that there is $u\in L^3_{\text{loc}}(Q_1)$, $p\in L^{3/2}_{\text{loc}}(Q_1)$, and $F\in L^3_{\text{loc}}(Q_1)$ and subsequences of $(u^{k})_{k\in \N}$, $(p^{k})_{k\in \N}$, and $(F^{k})_{k\in \N}$ (again labeled $u^k,p^k,F^k$) such that
$$
\begin{cases}
u^{k}\rightarrow u \; \text{in}\;  L^3_{\text{loc}}(Q_1)\\
p^{k}\rightharpoonup p \; \text{in}\;  L^{3/2}(Q_1)\\
F^{k}\rightarrow F \; \text{in}\;  L^3_{\text{loc}}(Q_1)
\end{cases}
$$
as $k\rightarrow \I.$ From \eqref{SeqBounds}, we may assume without any loss of generality that that 
$$
\begin{cases}
a^k \rightarrow a\\
b^k_j \rightarrow b_j
\end{cases}, \quad j=1,2,3
$$
in $\R^3$, as $k\rightarrow \infty$. 
\par It is follows from this convergence that $(u,p,F)$ are weak solutions of the following linear PDE
\begin{equation}\label{LIMITPDEupF}
\begin{cases}
\hspace{.15in}\partial_t u+ (a\cdot \nabla)u =\Delta u - \nabla p + \sum^3_{j=1}(b_j\cdot\nabla)F_j\\
\partial_t F_j + (a\cdot \nabla)F_j =\Delta F_j  + (b_j\cdot\nabla)u\\
\hspace{.8in}\nabla \cdot u=\nabla \cdot F_j=0
\end{cases}
\end{equation}
and satisfy the energy inequality 
{\small
\begin{align}\label{EnergyLimit}
\int_{B_1\times\{s\}}\phi\left(\frac{ |u|^2}{2}+\frac{ |F|^2}{2}\right)dx + \int^{s}_{-1/2}\int_{B_1}\phi\left(|\nabla u|^2+|\nabla F|^2\right)dxds \le  \int^{s}_{-1/2}\int_{B_1}\left[
\left(\frac{|u|^2}{2}+\frac{|F|^2}{2}\right)(\phi_t+\Delta \phi) \right. \nonumber \\
\left. + \left\{  -\sum^3_{j=1}(u\cdot F_j)b_j+\left(\frac{|u|^2}{2}+\frac{|F|^2}{2}\right)a+p u\right\}\cdot \nabla \phi \right]dxds
\end{align}}
for $\phi\in C^\I_c(B), \phi\ge 0$ and $s\in (-1/2,1/2).$ 

\par 3. We claim that there is a constant $C_1=C_1(M_0)$ such that 
\begin{equation}\label{LIMreguF}
\begin{cases}
|u(x,t)-u(y,s)|\le C_1(|x-y| +|t-s|^{1/3})\\
|F(x,t)-F(y,s)|\le C_1(|x-y| +|t-s|),
\end{cases}
\end{equation}
for $ (x,t),(y,s)\in Q_{1/2}$.  In particular, we are asserting that $u$ is H\"{o}lder continuous and $F$ is Lipschitz continuous on $Q_{1/2}.$  

\par To see
this, we first take the divergence of the first equation in \eqref{LIMITPDEupF} to get for almost every $t\in (-1/2,1/2)$ 
$$
-\Delta p(x,t)=0, \quad x\in B_1.
$$
Thus $p(t)\in C^\infty(B_1)$ for almost every $t\in (-1/2,1/2)$.  Recall the following estimate for harmonic functions: $-\Delta h=0$, in $B_1\subset \R^n$, then for each multiindex $\alpha$, 
$q\ge 1$, and $\theta \in (0,1)$, there is a $C(q,\alpha)$ such that
\begin{equation}\label{CriticalHarmonicEst}
|\partial^\alpha h|_{L^\I(B_\theta)}\le \frac{C(q,\alpha)}{(1-\theta)^{|\alpha|+n}}|h|_{L^q(B_1)}
\end{equation}
(see Theorem 7, page 29 of \cite{Evans}).  
As $p\in L^{3/2}(Q_1)$, we conclude that 
$$
\partial_x^\alpha p\in L^{3/2}((-1/2,1/2), L^\I_{\text{loc}}(B_1))
$$
for each multiindex $\alpha.$

\par Next, we take the curl of the first two equations in \eqref{LIMITPDEupF} to arrive at

\begin{equation}\label{CURLeqwrj}
\begin{cases}
\hspace{.02in}\partial_t w+ (a\cdot \nabla)w =\Delta w+ \sum^3_{j=1}(b_j\cdot\nabla)r_j\\
\partial_t r_j + (a\cdot \nabla)r_j =\Delta r_j  + (b_j\cdot\nabla)w
\end{cases}
\end{equation}
where of course
$$
w:=\nabla\times u\quad \text{and}\quad r_j:=\nabla\times F_j.
$$
From the local energy estimate for $(u,p,F)$ \eqref{EnergyLimit}, we have $w,r_j \in L^2_{\text{loc}}(Q_1)$.  As
$$
\begin{cases}
\hspace{.02in}\partial_t w -\Delta w = \nabla\cdot \left(\sum^3_{j=1}r_j\otimes b_j-w\otimes a\right) \in L^2((-1/2,1/2),H^{-1}(B_1))\\
\partial_t r_j -\Delta r_j =  \nabla\cdot\left( w\otimes b_j -r_j\otimes a\right) \in L^2((-1/2,1/2),H^{-1}(B_1))
\end{cases}
$$
It follows from standard energy estimates for the parabolic system \eqref{CURLeqwrj} that 
$$
w,r_j\in L^{\I}((-1/2,1/2), L^2(B_1))\cap L^{2}((-1/2,1/2), H^1(B_1)).
$$
As \eqref{CURLeqwrj} is linear,  and in particular each of the derivatives $\partial_x^\alpha w,\partial_x^\alpha r_j$ also satisfy \eqref{CURLeqwrj}, we  conclude by induction 
that for each multiindex $\alpha$
$$
\partial_x^\alpha w,\partial_x^\alpha r_j\in L^{\I}((-1/2,1/2), L^2(B_1))\cap L^{2}((-1/2,1/2), H^1(B_1)).
$$
Therefore,  each space variable of both $w$ and $r_j$ is locally bounded on $Q_1.$
\par Fom the condition $\nabla\cdot u=0$,  we have
\begin{equation}\label{CURLID}
-\Delta u=\nabla\times w. 
\end{equation}
In particular, the Biot-Savart law reads for for each compact $G\subset B_1$,
$$
u(x,t)=\int_{G}\nabla \phi(x-y)\times w(y,t)dy+ A(x,t), \quad (x,t)\in G\times (-1/2,1/2)
$$
where $G\ni x\mapsto A(x,t)$ is harmonic for almost each $t\in (-1/2,1/2)$ (see Lemma 2 of \cite{Serrin} for a proof).  From the local energy estimate for the solution $(u,p,F)$ \eqref{EnergyLimit}, we know that the $L^2(G)$ norm 
of $u$ is bounded independently of $t$ belonging to compact subintervals of $(-1/2,1/2)$. From our estimates on $w$, we conclude that the $L^2(G)$ norm of $A$ is also bounded independently of such $t$. As $A$ is harmonic, it is immediate from the mean value property that $A$ is locally bounded on $F\times I$, for each interval $I\subset (-1/2,1/2).$ Hence $u$ is locally bounded on $Q_1$; differentiating \eqref{CURLID} and recalling our estimates on $w$, we also see that  $\partial^\alpha_x u$ is locally bounded on $Q_1$.  

\par Arguing in the same manner, we have that $\partial^\alpha_xF_j$ is locally bounded on $Q_1$ for each multiindex $\alpha$ and $j=1,2,3$.  It is now immediate from equation \eqref{LIMITPDEupF} that 
$$
\partial_t u= -(a\cdot \nabla)u +\Delta u - \nabla p + \sum^3_{j=1}(b_j\cdot\nabla)F_j\in L^{3/2}_{\text{loc}}((-1/2,1/2), L^\infty_{\text{loc}}(B_1))
$$
and 
$$
\partial_t F_j =- (a\cdot \nabla)F_j +\Delta F_j  + (b_j\cdot\nabla)u\in L^{\infty}_{\text{loc}}((-1/2,1/2), L^\infty_{\text{loc}}(B_1))
$$
from which the claimed estimates \eqref{LIMreguF} readily follow.

\par 4. Therefore we have that there is constant $C_2=C_2(M_0)$ such that 
$$
\left(\fthetaint |u-u_{Q_{\theta_0}}|^3dyds\right)^{1/3}+\left(\fthetaint |F-F_{Q_{\theta_0}}|^3dyds\right)^{1/3}\le C_2 \theta_0^{2/3}.
$$
By the convergence in $L^3_{\text{loc}}(Q_1)$ of $(u^k,F^k)$ to $(u,F)$ as $k\rightarrow \infty$,
\begin{equation}\label{uandfcont}
\left(\fthetaint |u^k-u^k_{Q_{\theta_0}}|^3dyds\right)^{1/3}+\left(\fthetaint |F^k-F^k_{Q_{\theta_0}}|^3dyds\right)^{1/3}\le 2 C_2 \theta_0^{2/3}
\end{equation}
for all $k\in \N$ sufficiently large. 

\par  Let us now derive an estimate on the average of the $p^k$.  Direct computation from \eqref{PDEsequence} gives 
$$
-\Delta p^k=\epsilon_k \nabla \cdot [(u^k\cdot \nabla)u^k - \sum^3_{j=1}(F^k_j\cdot \nabla)F^k_j],\quad x\in B_1
$$
for for almost every $ s\in (-1/2,1/2)$. For each $ s\in (-1/2,1/2)$ that this equation holds weakly on $B_1$, let $g^k(s)$ denote the unique solution of the PDE
$$
\begin{cases}
-\Delta g^k=\epsilon_k \nabla \cdot [(u^k(s)\cdot \nabla)u^k(s) - \sum^3_{j=1}(F^k_j(s)\cdot \nabla)F^k_j(s)], \quad x\in B_{3/4}\\
\hspace{.25in} g^k=0, \quad x\in \partial B_{3/4}
\end{cases}
$$
and set
$$
h^k(s):=p^k(s)-g^k(s).
$$
Clearly, $h^k(s)$ is harmonic on $B_{3/4}$.  Virtually the same argument that lead us to the bound \eqref{pfivethirdsEst} provides the estimate
$$
\iint_{Q_{1/2}}|g^k|^{3/2}dyds\le C\epsilon_k
$$ 
where $C$ depends (say) only on an upper bound for the integrals $\iint_{Q_{5/6}}\left(|\nabla u^k|^2+|\nabla F^k|^2\right)dyds$.  The local energy estimate
 for the solution $(u^k,p^k,F^k)$ \eqref{EnergySequence} combined with the $L^3(Q_1)$ bounds \eqref{IntegralboundsSequence} for $(u^k,F^k)$ assures us that these integrals are uniformly bounded above.  Therefore, by these remarks and the estimate \eqref{CriticalHarmonicEst} 
 
\begin{eqnarray}
\theta_0 \left(\fthetaint |p^k-p^k_{B_{\theta_0}}|^{3/2}\right)^{2/3}&\le & \theta_0 \left(\fthetaint |g^k-g^k_{B_{\theta_0}}|^{3/2}\right)^{2/3}+ \theta_0 \left(\fthetaint |h^k-h^k_{B_{\theta_0}}|^{3/2}\right)^{2/3} \nonumber \\
&\le & 2 \theta_0\left(\fthetaint |g^k|^{3/2}\right)^{2/3}+ C\theta_0 \left(\frac{1}{\theta_0^5}\int^{\theta_0^2/2}_{-\theta_0^2/2}\int_{B_{\theta_0}}|h-h_{B_{\theta_0}}|^{3/2}dyds \right)^{2/3} \nonumber \\
&\le & \frac{2\theta_0}{\theta_0^{10/3}}\left(\iint_{Q_{\theta_0}} |g^k|^{3/2}\right)^{2/3}+ C\theta_0 \left(\frac{1}{\theta_0^5}\int^{\theta_0^2/2}_{-\theta_0^2/2}\int_{B_{\theta_0}}(\theta_0 |\nabla h^k(s)|_{L^\I(B_{\theta_0})})^{3/2}ds \right)^{2/3}  \nonumber \\
&\le &  \frac{2\theta_0}{\theta_0^{10/3}}\left(\iint_{Q_{1/2}} |g^k|^{3/2}\right)^{2/3}+ C\theta_0^{2/3} \left(\int^{\theta_0^2/2}_{-\theta_0^2/2}|\nabla h^k(s)|_{L^\I(B_{\theta_0})}^{3/2}ds \right)^{2/3}  \nonumber \\
&\le &  \frac{C \theta_0\epsilon_k}{\theta_0^{10/3}} + C\theta_0^{2/3} \left(\int^{\theta_0^2/2}_{-\theta_0^2/2}| h^k(s)|_{L^\I(B_{5/8})}^{3/2}ds \right)^{2/3}  \nonumber\\
&\le & \frac{C \epsilon_k}{\theta_0^{7/3}} + C\theta_0^{2/3} \left(\iint_{Q_{5/8}}|h^k|^{3/2}dyds \right)^{2/3}  \nonumber\\
&\le & \frac{C \epsilon_k}{\theta_0^{7/3}} + C\theta_0^{2/3} \left(\iint_{Q_{5/8}}(|g^k|^{3/2}+|g^k|^{3/2})dyds \right)^{2/3}  \nonumber \\
&\le & \frac{C \epsilon_k}{\theta_0^{7/3}}+ C\theta_0^{2/3}.  \nonumber
\end{eqnarray}
Consequently, for $k$ sufficiently large 
\begin{equation}\label{PressureCont}
\theta_0\left(\fthetaint|p^k-p^k_{B_{\theta_0}}|^{3/2}dyds\right)^{2/3}\le C_3 \theta_0^{2/3}
\end{equation}
for some universal constant $C_3$ (since for all large $k$ we have $\epsilon_k\le \theta_0^3)$.  Combining \eqref{uandfcont} and \eqref{PressureCont}, it is readily checked that a contradiction to \eqref{WhatWeDontWant} is obtained, for all $k$ sufficiently large, by choosing 
$$
c_1(M_0):=2(2C_2(M_0)+C_3), \quad M_0>0.
$$
Moreover, our choice $c_1$ is independent of $\theta$. 
\end{proof}

\section{A local criterion for H\"{o}lder continuity}\label{HC}
With Lemma \ref{DecayLem} in hand, we are now in position iterate its conclusion which is the major step in establishing a local criterion for H\"{o}lder continuity
for weak solutions of the system of PDE \eqref{System1}.  As a corollary of this H\"{o}lder continuity criterion, we prove Theorem \ref{likeSerrin} which is reminiscent of Serrin's higher regularity result for weak solutions of the incompressible Navier-Stokes equation \cite{Serrin}. 
\begin{lem}
Assume $(u,p,F)$ is weak solution on $Q_1$ and fix 
\begin{equation}\label{cthetabeta}
0<\theta<1/2, \quad M>0, \quad 0<\beta <2/3
\end{equation}
so that 
$$
c_1(M)\theta^{1/3-\beta/2}\le 1.
$$
There is $\epsilon_2>0$ such that if 

$$
\begin{cases}
(i)\quad Q_r(x,t)\subset Q_1, \; r\le R_1\\\\
(ii)\quad |u_{Q_r(x,t)}|+|F_{Q_r(x,t)}|<\frac{1}{2}M\\\\
(iii)\quad E(x,t,r)\le \epsilon_2
\end{cases}
$$
then for any $k=0,1,2,3,\dots$

$$
\begin{cases}
(i) \quad |F_{Q_{\theta^kr}(x,t)}|+ |u_{Q_{\theta^kr}(x,t)}|\le M\\\\
(ii) \quad E(x,t,\theta^kr)\le \epsilon_1\\\\
(iii) \quad E(x,t,\theta^{k+1}r)\le \theta^{(k+1)\beta}(1-\theta^{1/3-\beta/2})^{-1}E(x,t,r)
\end{cases}
$$
\end{lem}

\begin{proof}
1. First set 
$$
I(x,t,r):=\left(\ffint  |u-u_{Q_r}|^3dyds\right)^{1/3}+\left(\ffint  |F-F_{Q_r}|^3dyds\right)^{1/3}
$$
for $Q_r=Q_r(x,t)$, and notice that 
$$
|u_{Q_r(x,t)}-u_{Q_{\theta r}(x,t)}|+|F_{Q_r(x,t)}-F_{Q_{\theta r}(x,t)}|\le \frac{I(x,t,r)}{\theta^{5/3}}.
$$
The above estimate follows directly from H\"{o}lder's inequality, and using the triangle inequality
$$
|u_{Q_{\theta^kr}(x,t)}|+|F_{Q_{\theta^kr}(x,t)}|  \le \frac{1}{\theta^{5/3}}\sum^{k-1}_{j=0}I(x,t,\theta^j r)+ |u_{Q_r(x,t)}|+|F_{Q_r(x,t)}|.
$$
We define 
$$
\epsilon_2:=(1-\theta^{1/3-\beta/2})\min\left\{\frac{1}{2}\epsilon_1, \theta^{5/3}(1-\theta^\beta)\frac{M}{2}\right\}
$$
as our candidate for $\epsilon_2$ described in the statement of the lemma. We argue by induction below. 

\par 2. We establish the claim for $k=0$.  By assumption and our choice of $\epsilon_2$ 
$$
E(x,t,r)\le \epsilon_2\le (1-\theta^{1/3-\beta/2})\frac{\epsilon_1}{2}\le \frac{\epsilon_1}{2}\le \epsilon_1.
$$
It is also easy to verify that the conditions of the Lemma \ref{DecayLem} are satisfied, and thus

\begin{eqnarray}
E(x,t,\theta r) & \le & c_1\theta^{2/3} E(x,t,r) \nonumber \\
                       & \le & c_1\theta^{1/3-\beta/2}\theta^{1/3+\beta/2}E(x,t,r) \nonumber \\
                       & \le & 1 \cdot \theta^{1/3+\beta/2}E(x,t,r) \nonumber \\
                       &\le & \theta^\beta(1-\theta^{1/3-\beta/2})^{-1}E(x,t,r) \nonumber
\end{eqnarray}
as $\theta^{1/3+\beta/2}\le \theta^\beta(1-\theta^{1/3-\beta/2})^{-1}$ (which happens if and only if $\theta^{1/3-\beta/2}\le 2$).  

\par 3. Assume for $s=0,1,\dots, k$ 

$$
\begin{cases}
(i)_s \quad |u_{Q_{\theta^sr}(x,t)}|+|F_{Q_{\theta^sr}(x,t)}|\le M\\\\
(ii)_s \quad E(x,t,\theta^sr)\le \epsilon_1\\\\
(iii)_s \quad E(x,t,\theta^{s+1}r)\le \theta^{(s+1)\beta}(1-\theta^{1/3-\beta/2})^{-1}E(x,t,r)
\end{cases}
$$
We show that these assumptions imply that the above bounds hold in turn for $s=k+1.$ For simplicity, we suppress the $(x,t)$ dependence in our arguments below.

\par For $(i)_{k+1}$:
\begin{eqnarray}
|u_{r\theta^{k+1}}|+|F_{r\theta^{k+1}}|&\le & \frac{1}{\theta^{5/3}}\sum^{k}_{j=0}I(\theta^j r) +|u_{Q_r}|+|F_{Q_r}|  \nonumber \\ 
                                                        &< &  \frac{1}{\theta^{5/3}}\sum^{k}_{j=0}E(\theta^j r) + \frac{M}{2} \nonumber \\ 
                                                        &\le & \frac{1}{\theta^{5/3}}(1-\theta^{1/3-\beta/2})^{-1}E(r)\sum^k_{j=0}\theta^{\beta j} + \frac{M}{2} \nonumber \\
                                                        &\le & \frac{1}{\theta^{5/3}}\frac{E(r)}{(1-\theta^{1/3-\beta/2}) (1-\theta^\beta)} + \frac{M}{2} \nonumber\\
                                                        &\le & \frac{\epsilon_2}{\theta^{5/3}(1-\theta^{1/3-\beta/2}) (1-\theta^\beta)} +\frac{M}{2} \nonumber \\
                                                        &\le & \frac{M}{2}+\frac{M}{2} \nonumber \\
                                                        &=&M.\nonumber
\end{eqnarray}
For $(ii)_{k+1}$: By assumption $(iii)_k$, we have 

\begin{eqnarray}
E(\theta^{k+1}r)&\le &\theta^{(k+1)\beta}(1-\theta^{1/3-\beta/2})^{-1}E(r) \nonumber \\
			&\le &\theta^{(k+1)\beta}(1-\theta^{1/3-\beta/2})^{-1}\epsilon_2 \nonumber \\
			&\le &\theta^{(k+1)\beta}\frac{\epsilon_1}{2},\nonumber \\
			&< &\epsilon_1. \nonumber
\end{eqnarray}
For $(iii)_{k+1}$: Observe that the hypotheses of the Lemma \eqref{DecayLem} are satisfied with $r\mapsto \theta^kr$.   Moreover, using \eqref{cthetabeta}
\begin{eqnarray}
E(\theta^{k+2}r) & = & E(\theta\theta^{k+1}r) \nonumber \\
			     & \le & c_1\theta^{2/3} E(\theta^{k+1}r) \nonumber \\
			     &\le & c_1 \theta^{2/3} \theta^{(k+1)\beta}(1-\theta^{1/3-\beta/2})^{-1}E(r)\nonumber\\
			     &= & c_1\theta^{1/3-\beta/2}\theta^{1/3+\beta/2}\theta^{(k+1)\beta}(1-\theta^{1/3-\beta/2})^{-1}E(r)\nonumber\\
			     &\le & 1\cdot \theta^{1/3+\beta/2}\theta^{(k+1)\beta}(1-\theta^{1/3-\beta/2})^{-1}E(r)\nonumber\\
			     &\le & \theta^{(k+2)\beta}(1-\theta^{1/3-\beta/2})^{-1}E(r)\nonumber.
\end{eqnarray}
The last inequality follows since $\beta/2 +1/3\ge \beta$ and $\theta\in (0,1)$ which trivially implies $\theta^{1/3+\beta/2}\le \theta^\beta.$

\end{proof}

The main use of the decay and iteration lemmas is the following proposition.  We omit the proof as it follows from the above lemma and relatively standard manipulations. See Proposition 2.8 in \cite{LS} for a related result, concerning the incompressible Navier-Stokes, and its proof, which is easily adapted to our framework.  

\begin{prop}\label{FirstHolder}
Assume that $(u,p,F)$ is a weak solution of \eqref{System1}. There are universal constants $\epsilon_3, R_2$ such that if 
$$
\begin{cases}
Q_r=Q_r(x,t)\subset Q_1, \quad r\le R_2\\
\overline{E}(x,t,r)<\epsilon_3
\end{cases}
$$
then $u$ and $F$ are H\"{o}lder continuous in some neighborhood of $(x,t)$. Here 
$$
\overline{E}(x,t,r):=\left(\ffint|u|^3\right)^{1/3} + r\left(\ffint|p|^{3/2}\right)^{2/3}+ \left(\ffint|F|^3\right)^{1/3}.
$$
\end{prop}
The scaling invariance properties of the system \eqref{System1} imply the following improvement of
Proposition \eqref{FirstHolder}.\footnote{Recall \eqref{ScaleInv}.}  This simple observation is as important as any we make in this work. 

\begin{cor}
Assume $(u,p,F)$ is a weak solution on $Q_1$ and let $\epsilon_3, R_2$ be as in Proposition \ref{FirstHolder}. Further suppose that 
$$
\begin{cases}
Q_r(x,t)\subset Q_1, \quad r\le R_2\\
\frac{r}{R_2} \overline{E}(x,t,r)<\epsilon_3
\end{cases}.
$$
Then $(x,t)$ is a regular point for $(u,p,F).$
\end{cor}

\begin{proof}
Set $\lambda=r/R_2$ and define
$$
\begin{cases}
u^\lambda(y,s):=\lambda u(x+\lambda y,t+\lambda^2 s)\\
u^\lambda(y,s):=\lambda^2 p(x+\lambda y,t+\lambda^2 s)\\
F^\lambda(y,s):=\lambda F(x+\lambda y,t+\lambda^2 s)
\end{cases}, (y,s)\in Q_{R_2}(0,0).
$$
It is immediate that $(u^\lambda,p^\lambda,F^\lambda)$ is a weak solution on $Q_{R_2}(0,0)$, and direct computation yields
\begin{eqnarray}
\overline{E}(0,0,R_2; u^\lambda,p^\lambda,F^\lambda)&=&\lambda \overline{E}(t,x,r; u,p,F)\nonumber\\
&=& \frac{r}{R_2} \overline{E}(t,x,r; u,p,F) \nonumber\\
&<&\epsilon_3. \nonumber
\end{eqnarray}
By  Proposition \ref{FirstHolder}, $u^\lambda$ and $F^\lambda$ are H\"{o}lder continuous in a neighborhood of $(0,0)$. Consequently, $u$ and $F$ are H\"{o}lder continuous in a neighborhood of 
$(x,t)$.
\end{proof}

\begin{cor}\label{PfivethirdsEst}
Assume $(u,p,F)$ is a weak solution on $Q_1$ and let $\epsilon_3, R_2$ be as in Proposition \ref{FirstHolder}. If
$$
\liminf_{r\rightarrow 0^+} r\overline{E}(x,t,r)< \frac{1}{2}\epsilon_3 R_2
$$
then $(x,t)$ is a regular point point for $(u,p,F)$.   In particular, there is $\epsilon_4>0$ such that the same conclusion holds provided
\begin{equation}\label{FirstCKNLimit}
\liminf_{r\rightarrow 0^+}\frac{1}{r^2}\iint_{Q_r{(x,t)}}\left\{ |u|^3+|p|^{3/2}+|F|^3\right\}dyds<\epsilon_4.
\end{equation}
\end{cor}

\begin{proof}
If 
$$
\liminf_{r\rightarrow 0^+} r\overline{E}(x,t,r)< \frac{1}{2}\epsilon_3 R_2,
$$
then $\inf_{0<r< \delta} r\overline{E}(x,t,r)< \frac{1}{2}\epsilon_3 R_2$ for any $\delta>0$.  Moreover, there is $r<R_2$ such that $\frac{r}{R_2}E(x,t,r)<\epsilon_3$ and $Q_r(x,t)\subset Q_1.$ The previous corollary then applies.  The second assertion follows immediately from the first. 
\end{proof}
The following claim is in the spirit of Serrin's regularity criterion for weak solutions of the incompressible Navier-Stokes equations. Note, however, that our result requires integrability of the pressure as well as integrability of the velocity and the deformation. 
\begin{thm}\label{likeSerrin}
Assume that $(u,p,F)$ is a weak solution on $Q_1$ such that
\begin{equation}\label{FirstCKNBound}
\int^{1/2}_{-1/2}\left(\int_{B_1}\left\{|u|^s +|p|^{s/2}+|F|^{s}\right\}dx\right)^{s'/s}dt<\infty.
\end{equation}
for some  $ s'\ge s$ satisfying
$$
\frac{3}{s}+\frac{2}{s'}\le 1
$$
Then $u,F$ is regular on $Q_1$.
\end{thm}

\begin{proof}
For each point $(x,t)\in Q_1$, the bound \eqref{FirstCKNBound} implies 
$$
\lim_{r\rightarrow 0}\frac{1}{r^2}\iint_{Q_r{(x,t)}}\left\{ |u|^3+|p|^{3/2}+|F|^3\right\}dyds=0,
$$
which in turn implies \eqref{FirstCKNLimit}. Therefore, the claim follows from the previous corollary. 
\end{proof}

\section{``A-B-C-D" estimates}\label{ABCDest}
In the previous section, we deduced that there is an $\epsilon_5>0$ such that if
\begin{equation}\label{KeyLimit}
\liminf_{r\rightarrow 0^+} r\overline{E}(x,t,r)< \epsilon_5,
\end{equation}
then $u$ and $F$ are both H\"{o}lder continuous near $(x,t)$.\footnote{$\epsilon_5:=\frac{1}{2}\epsilon_3 R_2$}  We claim that the central hypothesis of Theorem \ref{mainthm} implies the above limit. More precisely, we
assert the following fundamental proposition.
\begin{prop}\label{FundProp} Let $(u,p,F)$ be a weak solution on $Q_1$.  There is a universal constant $\epsilon>0$ such that inequality \eqref{KeyLimit} holds whenever 
$$
\limsup_{r\rightarrow 0^+}\frac{1}{r}\iint_{Q_r(x,t)}\left\{|\nabla u|^2 + |\nabla F|^2\right\}dyds<\epsilon.
$$
\end{prop}
In view of  Corollary \ref{PfivethirdsEst}, a proof of Proposition \ref{FundProp} establishes Theorem \ref{mainthm}.   To this end, we shall need three estimates involving the following integral quantities. For $(x,t)\in Q_1$, and 
$r>0$ so small that $Q_r(x,t)\subset Q_1$, we define

$$
\begin{cases}
\quad A(x,t,r):=\sup_{|t-s|\le r^2/2}\frac{1}{r}\int_{B_r(x)}\{|u(y,s)|^2+|F(y,s)|^2\}dy\\\\
\quad B(x,t,r):=\frac{1}{r}\iint_{Q_r(x,t)}\left\{|\nabla u(y,s)|^2 + |\nabla F(y,s)|^2\right\}dyds\\\\
\quad C(x,t,r):=\frac{1}{r^2}\iint_{Q_r(x,t)}\{|u(y,s)|^3+|F(y,s)|^3\}dyds\\\\
\quad D(x,t,r):=\frac{1}{r^2}\iint_{Q_r(x,t)}|p(y,s)|^{3/2}dyds\\
\end{cases}.
$$
Our arguments below are independent of $(x,t)$, so without any loss of generality we establish our results for $(x,t)=(0,0)$. For ease of notation, we also write $A(r):=A(0,0,r), B(r):=B(0,0,r), C(r):=C(0,0,r),$ and 
$D(r):=D(0,0,r)$ for $0<r\le 1.$  

\par We will need three estimates before undertaking the proof of Proposition \ref{FundProp}. They are very similar to the sequence of Lemmas needed in \cite{L} and as in the 
previous section we follow the path of \cite{LS} closely.  The statement of the lemmas in fact are nearly identical to the ones use to prove the version of Proposition \ref{FundProp} in \cite{LS}, but unfortunately 
the proofs had to be modified. Nevertheless, we would like to emphasize that the work of \cite{L} and \cite{LS} served as a guiding light for this section.

\begin{lem}\label{FinalCEst}  There is a universal constant $c>0$ such that
$$
C(r)\le c\left\{ \left(\frac{r}{\rho}\right)^3 A^{3/2}(\rho) +\left(\frac{\rho}{r}\right)^3 A^{3/4}(\rho) B^{3/4}(\rho) \right\},
$$
for $0<r\le \rho$.
\end{lem}

\begin{proof}
This assertion follows directly from following the well known inequality established in \cite{L} (Lemma 2.1). There is a universal constant $c>0$ such that
$$
C_0(r)\le c\left\{ \left(\frac{r}{\rho}\right)^3 A_0^{3/2}(\rho) +\left(\frac{\rho}{r}\right)^3 A_0^{3/4}(\rho) B_0^{3/4}(\rho) \right\},\quad 0<r\le \rho\le 1
$$
for all 
$$
v\in H^1(Q_1)\cap L^\I((-1/2,1/2); L^2(B_1)),
$$
where
$$
\begin{cases}
\quad A_0(r):=\sup_{|t|\le r^2/2}\frac{1}{r}\int_{B_r}|v(y,s)|^2dy\\
\quad B_0(r):=\frac{1}{r}\iint_{Q_r}|\nabla v(y,s)|^2dyds\\
\quad C_0(r):=\frac{1}{r^2}\iint_{Q_r}|v(y,s)|^3dyds
\end{cases}.
$$
\end{proof}

\begin{lem}\label{FinalABEst} There is a universal constant $c>0$ such that 
$$
A(r/2)+B(r/2)\le c\left\{C^{2/3}(r) +C^{1/3}(r) D^{2/3}(r)+A^{1/2}(r)B^{1/2}(r)C^{1/3}(r)\right\},
$$
for $0<r\le 1.$
\end{lem}
\begin{rem}
Recall the energy inequality \eqref{NSEenergy}: $\phi\in C^\I_c(Q_1), \phi\ge 0$
\begin{align*}
\int_{B_1\times\{t\}}\phi\left(\frac{|u|^2}{2}+\frac{|F|^2}{2}\right)dx &+ \int^t_{-1/2}\int_{B_1}\phi\left(|\nabla u|^2+|\nabla F|^2\right)dxds \le  \int^t_{-1/2}\int_{B_1}\left\{(\phi_t+\Delta \phi)\left(\frac{|u|^2}{2}+\frac{|F|^2}{2}\right)\right.\nonumber\\
&\left.  - F^tu\cdot F^t\nabla\phi+ \left(\frac{|u|^2}{2}+\frac{|F|^2}{2}+p\right)u\cdot \nabla \phi \right\}dxds
\end{align*}
$t\in [-1/2,1/2]$. As 
$$
\nabla\cdot u=0\quad\text{and}\quad \nabla\cdot F^t=0,
$$
we can replace the right hand side of the above inequality by 
{\small
$$
 \int^t_{-1/2}\int_{B_1}\left\{(\phi_t+\Delta \phi)\left(\frac{|u|^2}{2}+\frac{|F|^2}{2}\right)+a(t)\cdot F^t\nabla\phi- F^tu\cdot F^t\nabla\phi+ \left(\frac{|u|^2}{2}+\frac{|F|^2}{2}-b(t)+p\right)u\cdot \nabla \phi \right\}dxds
$$}
for any  $a\in L^1_{\text{loc}}(-1/2,1/2; \R^{3})$ and $b\in L^1_{\text{loc}}(-1/2,1/2)$.  In a crucial step below, we will use 
\begin{equation}\label{choiceAb}
\begin{cases}
a(t):= [F^tu]_{B_r}(t)= \gfint \; F^t(x,t)u(x,t)dx\\
b(t):= \frac{1}{2}|u|^2_{B_r}(t)+\frac{1}{2}|F|^2_{B_r}(t)= \gfint\left(\frac{|u(x,t)|^2}{2}+\frac{|F(x,t)|^2}{2}\right)dx
\end{cases}.
\end{equation}
\end{rem}

\begin{proof}
Choose $\phi \in C^\I_c(Q_1)$, nonnegative, such that $\phi\equiv 1$ on $Q_{1/2}$, and set 
$$
\phi_r(x,t):=\phi\left(\frac{x}{r},\frac{t}{r^2}\right), \quad (x,t)\in Q_r.
$$
Clearly $\phi_r\in C^\I_c(Q_r)$ and
$$
\begin{cases}
|\nabla \phi_r|\le c/r\\
|\partial_t\phi_r+\Delta \phi_r|\le c/r^2
\end{cases}
$$
for a constant $c\ge 0$ independent of $r\in (0,1).$ 

\par Using $\phi_r$ as the test function in the energy inequality \eqref{NSEenergy} and choosing the mappings $a$ and $b$ as in \eqref{choiceAb} gives 
\begin{align}\label{ABest1}
A(r/2)+B(r/2)&\le c\left\{ \frac{1}{r^3} \iint_{Q_r}(|u|^2+|F|^2) + \frac{1}{r^2}\iint_{Q_r}||u|^2-|u|^2_{B_r}||u| \right. \nonumber \\
&  +\frac{1}{r^2}\iint_{Q_r}||F|^2-|F|^2_{B_r}||u| +  \frac{1}{r^2}\iint_{Q_r}||F^tu|-[F^tu]_{B_r}||F| \nonumber \\
                      & \left. +\frac{1}{r^2}\left(\iint_{Q_r}|u|^3\right)^{1/3} \left(\iint_{Q_r}|p|^{3/2}\right)^{2/3}\right\}.
\end{align}
We now proceed to estimate each integral in the inequality above \eqref{ABest1} in terms of $A,B,C,D$. First note that

\begin{align*}
\frac{1}{r^3}\iint_{Q_r}(|u|^2+|F|^2) &\le c \frac{1}{r^3} \left(\iint_{Q_r}|u|^3+|F|^3\right)^{2/3} |Q_r|^{1/3} \\
                                              &\le c r^{5/3-3} \left(\iint_{Q_r}(|u|^3+|F|^3)\right)^{2/3}\\
                                              &= c C(r)^{2/3},
\end{align*}
for some universal constant $c>0$. Next, observe
\begin{align*}
\frac{1}{r^2}\left(\iint_{Q_r}|u|^3\right)^{1/3} \left(\iint_{Q_r}|p|^{3/2}\right)^{2/3} & = 
\left(\frac{1}{r^2}\iint_{Q_r}|u|^3\right)^{1/3} \left(\frac{1}{r^2}\iint_{Q_r}|p|^{3/2}\right)^{2/3} \\
& \le C(r)^{1/3}D^{2/3}(r).
\end{align*}
We also have, by employing H\"{o}lder's and Poincar\'{e}'s inequality, 

\begin{align*}
\iint_{Q_r}||u|^2-|u|^2_{B_r}||u| & \le \int^{r^2/2}_{-r^2/2}\left(\int_{B_r}||u|^2-|u|^2_{B_r}|^{3/2}\right)^{2/3}\left(\int_{B_r}|u|^3\right)^{1/3}dt \\
                                                       & \le   \int^{r^2/2}_{-r^2/2}\left(\int_{B_r}|u||\nabla u|\right) \left(\int_{B_r}|u|^3\right)^{1/3}dt \\
                                                       &\le \int^{r^2/2}_{-r^2/2} \left(\int_{B_r}|u|^2\right)^{1/2}\left(\int_{B_r}|\nabla u|^2\right)^{1/2} \left(\int_{B_r}|u|^3\right)^{1/3}dt\\
                                                       &\le c r^{1/2}A(r)^{1/2} \int^{r^2/2}_{-r^2/2}\left(\int_{B_r}|\nabla u|^2\right)^{1/2} \left(\int_{B_r}|u|^3\right)^{1/3}dt\\
                                                       &\le c r^{1/2}A(r)^{1/2} \left(\iint_{Q_r}| u|^3\right)^{1/3} \left(\int^{r^2/2}_{-r^2/2}\left(\int_{B_r}|\nabla u|^2\right)^{3/4}dt\right)^{2/3}\\
                                                       &\le c r^{1/2+2/3}A(r)^{1/2}C^{1/3}(r)\left(\iint_{Q_r}|\nabla u|^2\right)^{1/2}r^{1/2}\\
                                                       &\le c r^{1+2/3}A(r)^{1/2}C^{1/3}(r)\left(\iint_{Q_r}|\nabla u|^2\right)^{1/2}\\
                                                       &\le cr^{1+2/3+1/2}A(r)^{1/2}C^{1/3}(r)B^{1/2}(r)\\
                                                       &\le cr^{2+1/6}A(r)^{1/2}C^{1/3}(r)B^{1/2}(r)\\
                                                       &\le cr^{2}A(r)^{1/2}C^{1/3}(r)B^{1/2}(r),
\end{align*}
 for some universal constant $c$ independent of $r\in (0,1]$.  Arguing in virtually the same manner shows
$$
\frac{1}{r^2}\iint_{Q_r}||F|^2-|F|^2_{B_r}||u| +  \frac{1}{r^2}\iint_{Q_r}||F^tu|-[F^tu]_{B_r}||F| \le cA(r)^{1/2}C^{1/3}(r)B^{1/2}(r).
$$
Substituting all of these bounds into inequality \eqref{ABest1} gives the desired estimate
$$
A(r/2)+ B(r/2)\le c\left\{ C^{2/3}(r) + C^{1/3}D^{2/3}(r)+  A(r)^{1/2}C^{1/3}(r)B^{1/2}(r)\right\}.
$$
\end{proof}


\begin{lem}\label{FinalDest} There is a universal constant $c>0$ such that
$$
D(r)\le c\left\{ \frac{r}{\rho}D(\rho) +\left(\frac{\rho}{r}\right)^2 A^{3/4}(\rho) B^{3/4}(\rho) \right\},
$$
$0<r\le \rho$.
\end{lem}

\begin{proof}
1. Taking the divergence of the first equation in \eqref{System1} gives 
\begin{equation}\label{NewPeq}
-\Delta p=\nabla\cdot\left[\nabla\cdot\left(uu^t -FF^t\right)\right]=\nabla\cdot\left[uu^t-[uu^t]_{\rho}-(FF^t-[FF^t]_{\rho})\right],
\end{equation}
as $(u\cdot \nabla)u=\nabla\cdot uu^t$.   Here $[uu^t]_{\rho}$ denotes the average of $uu^t$ on $B_\rho$ and likewise for $[FF^t]_{\rho}$. 

\par Now let  $Q_1=Q_1(t)\in W_0^{2,3/2}(B_\rho;\R^{3\times 3})$ denote the unique weak solution of the PDE
$$
\begin{cases}
-\Delta Q_1 = uu^t-[uu^t]_{\rho}-(FF^t-[FF^t]_{\rho}), &\quad x\in B_\rho \\
\hspace{.25in} Q_1 = 0, &\quad x\in \partial B_\rho 
\end{cases},
$$
for (almost every) $|t|<1/2$.  It is immediate that
$$
p_1=\nabla\cdot[\nabla \cdot Q_1]\in L^{3/2}(B_\rho)
$$
satisfies \eqref{NewPeq}, and by the Calder\'{o}n-Zygmund inequality
\begin{equation}\label{p1Ineq}
|p_1|_{L^{3/2}(B_\rho)}\le  |\nabla\cdot[\nabla \cdot Q_1]|_{L^{3/2}(B_\rho)}\le C (|uu^t- [uu^t]_{\rho}|_{L^{3/2}(B_\rho)}+|FF^t- [FF^t]_{\rho}|_{L^{3/2}(B_\rho)}).
\end{equation}
It is also clear that
$$
p_2:=p-p_1
$$
is harmonic in $B_\rho$, for almost every $|t|<1/2$. 

\par 2.  By the estimate \eqref{p1Ineq} and Poincar\'{e}'s inequality, 

\begin{align*}
\left(\int_{B_\rho}|p_1|^{3/2}dx\right)^{2/3} &\le c \int_{B_\rho}\{|u||\nabla u| + |F||\nabla F|\}dx\\
										      &\le c\left(\int_{B_\rho}(|u|^{2}+|F|^{2})dx\right)^{1/2} \left(\int_{B_\rho}(|\nabla u|^{2}+|\nabla F|^{2})dx\right)^{1/2}\\
										      &\le c\rho^{1/2} A^{1/2}(\rho) \left(\int_{B_\rho}(|\nabla u|^{2}+|\nabla F|^{2})dx\right)^{1/2}.
\end{align*}
Integrating time from $-\rho^2/2$ to $+\rho^2/2$, gives

\begin{align*}
\iint_{Q_\rho}|p_1|^{3/2}& \le c\rho^{3/4}A^{3/4}(\rho)\int^{\rho^2/2}_{-\rho^2/2}\left(\int_{B_\rho}(|\nabla u|^2+|\nabla F|^2)dx\right)^{3/4}\\
					   & \le c\rho^{3/4}A^{3/4} \left(\iint_{Q_\rho}(|\nabla u|^2+|\nabla F|^2)dxdt\right)^{3/4}\rho^{1/2}\\
					   & \le c\rho^{2}A^{3/4} B(\rho)^{3/4}.
\end{align*}
It also follows that 

\begin{align*}
\iint_{Q_\rho}|p_2|^{3/2} & \le \iint_{Q_\rho}|p|^{3/2} +\iint_{Q_\rho}|p_1|^{3/2}  \\
                                           &\le c\rho^2\left\{ D(\rho) +A^{3/4} B(\rho)^{3/4}\right\}.
\end{align*}
\par 3. $p_2$ is harmonic and thus $|p_2|^{3/2}$ is subharmonic for almost every $|t|<1/2$.  In particular, 
$$
r\rightarrow \frac{1}{r^3}\int_{B_r}|p|^{3/2}dx
$$
is nondecreasing.  Consequently, 
$$
\frac{1}{r^2}\iint_{Q_r}|p_2|^{3/2} \le \frac{r}{\rho}\cdot \frac{1}{\rho^2}\iint_{Q_\rho}|p_2|^{3/2} 
$$
for $0<r\le \rho.$ Combining the preceding inequalities yields
\begin{align*}
D(r)&= \frac{1}{r^2}\iint_{Q_r}|p|^{3/2}dxdt  \\
       &\le \frac{1}{r^2}\iint_{Q_r}|p_1|^{3/2}dxdt+\frac{1}{r^2}\iint_{Q_r}|p_2|^{3/2}dxdt \\
       &\le \frac{1}{r^2}\iint_{Q_\rho}|p_1|^{3/2}dxdt+\frac{r}{\rho}\cdot \frac{1}{\rho^2}\iint_{Q_\rho}|p_2|^{3/2}dxdt \\
       &\le \left(\frac{\rho}{r}\right)^2A^{3/4}(\rho)B^{3/4}(\rho)+c\frac{r}{\rho}\left[D(\rho)+A^{3/4} B(\rho)^{3/4}\right] \\
       &\le c\left\{ \frac{r}{\rho}D(\rho) + \left(\frac{\rho}{r}\right)^2A^{3/4} B(\rho)^{3/4}\right\},
\end{align*}
as desired.
\end{proof}
We are finally ready to prove of Proposition \ref{FundProp}.  We remark that this argument follows very closely with Proposition 2.9 in \cite{LS}. Nevertheless, we provide it here for completion.  
\begin{proof} (of Proposition \ref{FundProp})
1. Set 
$$
{\mathcal E}(r):=A^{3/2}(r)+D^2(r), \quad 0<r\le 1
$$
and observe
\begin{align}\label{rEboundbyEE}
r \overline{E}(r) &= r \left\{\left(\ffint |u|^3\right)^{1/3}+r\left(\ffint |p|^{3/2}\right)^{2/3}+\left(\ffint |F|^3\right)^{1/3}\right\}\nonumber \\
&\le c\left\{ C(r)^{1/3}+ D(r)^{2/3}\right\} \nonumber \\
&\le c\left\{C(r)+D^2(r)\right\}^{1/3} \nonumber \\
&\le c\left\{C(r)+{\mathcal E}(r)\right\}^{1/3}.
\end{align}
Recall that our aim is to show that there is $\epsilon>0$ such that
$$
\limsup_{r\rightarrow 0^+}B(r)<\epsilon\quad \Rightarrow \quad\liminf_{r\rightarrow 0^+}r \overline{E}(r)<\epsilon_5.
$$
Therefore, we assume that $\limsup_{r\rightarrow 0^+}B(r)<\epsilon$ and choose $\epsilon$ below to establish the above implication. 

\par 2.  Let $\theta,\rho \in (0,1/2)$ be fixed and observe that easy corollaries of Lemmas  \ref{FinalCEst}, \ref{FinalABEst}, and \ref{FinalDest} are
$$
\begin{cases}
C(\theta\rho) \le c\left\{\theta^3 A^{3/2}(\rho) +\theta^{-3} A^{3/4}(\rho) B^{3/4}(\rho) \right\},\\\\
A^{3/2}\left(\frac{1}{2}\theta \rho\right)  \le c\left\{ C(\theta \rho) +D^2(\theta\rho)+A^{3/2}(\theta\rho)B^{3/2}(\theta\rho) \right\},\\\\
D(\theta\rho) \le c\left\{ \theta D(\rho) +\theta^{-2} A^{3/4}(\rho) B^{3/4}(\rho) \right\}.
\end{cases}
$$
Using the inequalities (that are trivial to verify)
$$
A(\theta\rho)\le \frac{1}{\theta}A(\rho)\quad\text{and}\quad B(\theta\rho)\le \frac{1}{\theta}B(\rho),
$$
we also have 
\begin{align*}
A^{3/2}\left(\frac{1}{2}\theta \rho\right) &\le C\left\{ \theta^2D^2(\rho) + \theta^3A^{3/2}(\rho) + \frac{1}{\theta^3}A^{3/4}(\rho)B^{3/4}(\rho) +\left(\frac{1}{\theta^3}+\frac{1}{\theta^4}\right)A^{3/2}(\rho)B^{3/2}(\rho)\right\}
\end{align*}
and
$$
D^2\left(\frac{1}{2}\theta \rho\right) \le c\left\{\theta^2 D^2(\rho)+\frac{1}{\theta^4}A^{3/2}(\rho)B^{3/2}(\rho)\right\}.
$$
With the above estimates on $A^{3/2}$ and $D^2$, 
\begin{align}\label{KeyBoundEE}
{\mathcal E}\left(\frac{1}{2}\theta \rho\right) & = A^{3/2}\left(\frac{1}{2}\theta \rho\right) + D^2\left(\frac{1}{2}\theta \rho\right)\nonumber \\
&\le c\left\{ \theta^2D^2(\rho) + \theta^3A^{3/2}(\rho) +\frac{1}{\theta^3}A^{3/4}(\rho)B^{3/4}(\rho)  +\frac{1}{\theta^{4}}A^{3/2}(\rho)B^{3/2}(\rho)\right\} \nonumber\\
&\le c\left\{ \theta^2{\mathcal E}(\rho) + \frac{1}{\theta^4}A^{3/2}(\rho)B^{3/2}(\rho) +\frac{B^{3/2}(\rho)}{\theta^{9}}\right\} \nonumber\\
&\le c\left\{ \theta^2{\mathcal E}(\rho) + \left(\frac{1}{\theta^9}+\frac{1}{\theta^4}A^{3/2}(\rho) \right)B^{3/2}(\rho)\right\} \nonumber \\
&\le c\left\{ \theta^2{\mathcal E}(\rho) +\frac{1}{\theta^9}B^{3/2}(\rho)\right\}
\end{align}
as $A(\rho)\le A(1/2)\le c.$

\par 3. Now choose $\theta \in (0,1/2)$ so small that for each of the (universal) constants above 
\begin{equation}\label{chooseTHETA}
c\theta^2\le 1/2.
\end{equation}
We may also select $\delta_\epsilon\in (0,1/2)$ so small that 
$$
B(\rho)\le 2\epsilon\quad \text{for}\quad 0<\rho<\delta_\epsilon
$$
This can be done by hypothesis.  With $\theta\in (0,1/2)$ chosen to satisfy \eqref{chooseTHETA}, \eqref{KeyBoundEE} gives
$$
{\mathcal E}\left(\tau \rho\right) \le \frac{1}{2}{\mathcal E}(\rho) +c_\tau \epsilon^{3/2}, \quad 0<\rho<\delta_\epsilon
$$
where we have set
$$
\begin{cases}
\tau:=\frac{1}{2}\theta\\
c_\tau:=c2^{3/2}/\theta^9
\end{cases}.
$$
We interpret \eqref{decayEE} to be a decay estimate for ${\mathcal E}$ as a simple induction argument provides the sequence of inequalities
\begin{equation}\label{decayEE}
{\mathcal E}(\tau^k\rho)\le \frac{1}{2^{k+1}}{\mathcal E}(\rho) +c_\tau \epsilon^{3/2}\sum^{k}_{j=0}\frac{1}{2^j}\le \frac{1}{2^{k+1}}{\mathcal E}(\rho) +2c_\tau \epsilon^{3/2},
\end{equation}
for any fixed $\rho\in (0,\delta_\epsilon)$ and $k\in \N$.

\par 4. Employing \ref{FinalCEst} and \eqref{decayEE},

\begin{align*}
C(\tau^{k+1}\rho) & \le  c\left\{ \tau^3 A^{3/2}(\tau^k\rho) +\tau^{-3} A^{3/4}(\tau^k\rho) B^{3/4}(\tau^k\rho) \right\}\\
                                        &\le c\left\{ \tau^3 {\mathcal E}(\tau^k\rho) +\tau^{-3} {\mathcal E}^{1/2}(\tau^k\rho) (2\epsilon)^{3/4} \right\}\\ 
                                        &\le  c\left\{ \tau^3 \left(\frac{1}{2^{k}}{\mathcal E}(\rho) +2c_\tau \epsilon^{3/2} \right)  \right. \\
                                        & \left.\quad +\tau^{-3} \left(\left(\frac{1}{2^{k}}{\mathcal E}(\rho) +2c_\tau \epsilon^{3/2} \right) \right)^{1/2}\epsilon^{3/4} \right\}.
\end{align*}
In view of the above estimate and inequality \eqref{rEboundbyEE},
\begin{align*}
\limsup_{k\rightarrow \infty}\tau^k\rho \overline{E}(\tau^k\rho) &\le c\limsup_{k\rightarrow \infty} \left[C(\tau^k\rho)+{\mathcal E}(\tau^k\rho)\right]^{1/3} \\
&\le c\left\{ \tau^3 2c_\tau \epsilon^{3/2} +\tau^{-3} \left(2c_\tau \epsilon^{3/2} \right)^{1/2}\epsilon^{3/4} +2c_\tau \epsilon^{3/2}\right\}^{1/3}\\
&\le c\left\{ \tau^3 2c_\tau+\tau^{-3} \left(2c_\tau\right)^{1/2} +2c_\tau\right\}^{1/3}\epsilon^{1/2}.
\end{align*}
We conclude by choosing $\epsilon>0$ so small that 
$$
\limsup_{k\rightarrow \infty}\tau^k\rho \overline{E}(\tau^k\rho) <\frac{1}{2}\epsilon_5.
$$
\end{proof}

\begin{rem}
A close inspection of the work in this paper shows that the methods employed establish an analogous version of Theorem \ref{mainthm} for weak solutions of the system 
\begin{equation}
\begin{cases}
\hspace{.06in}\partial_tu + (u\cdot \nabla)u  = \Delta u - \nabla p +\nabla\cdot FF^t\\
\partial_t F +(u\cdot \nabla)F  = \Delta F+\nabla u F + (\nabla Q)^t\\
\hspace{.73in}\nabla \cdot u = 0\\
\hspace{.65in}\nabla\cdot F^t=0
\end{cases},
\end{equation}
where $Q=Q(x,t)\in \R^3$.
\end{rem}

\appendix

\newpage


\end{document}